\documentclass[12pt]{amsart}
\usepackage{txfonts}
\usepackage{amssymb}

\theoremstyle{plain}

\newtheorem{Thm}{Theorem}

\newtheorem{Lem}[Thm]{Lemma}
\newtheorem{Def}[Thm]{Definition}
\newtheorem{Rk}[Thm]{Remark}

\begin{document}

\title[Global Yamabe flow ]
{Global Yamabe flow on asymptotically flat manifolds}

\author{Li Ma}
\address{Li MA,  School of Mathematics and Physics\\
  University of Science and Technology Beijing \\
  30 Xueyuan Road, Haidian District
  Beijing, 100083\\
  P.R. China }

\address{ Department of Mathematics \\
Henan Normal university \\
Xinxiang, 453007 \\
China}

\thanks{Li Ma's research was partially supported by the National Natural
  Science Foundation of China (No.11771124)}

\begin{abstract}
In this paper, we study the existence of global Yamabe flow on asymptotically flat (in short, AF or ALE) manifolds. Note that the ADM mass is preserved in dimensions 3,4 and 5. We present a new general local existence of Yamabe flow on a complete Riemannian manifold with the initial metric quasi-isometric to a
background metric of bounded scalar curvature.
Asymptotic behaviour of the Yamabe flow on ALE manifolds is also addressed provided the initial scalar curvature is non-negative and there is a bounded subsolution to the corresponding Poisson equation. We also present a maximum principle for a very general parabolic equations on the complete Riemannian manifolds.


\end{abstract}

\maketitle

\section{Introduction}\label{sect1}
The goal of this paper is to  consider the global Yamabe flow on complete manifolds. This topic has recently been studied in the works \cite{M}, \cite{MC}, \cite{CZ} and \cite{M1,M2}. Since the Yamabe flow is degenerate, the expected global flow is rare, however, the Yamabe flow on  asymptotically flat (in short, AF or ALE) manifolds is widely believed to be global. We shall confirm this in this paper.
Hamilton  \cite{Hamilton1989} \cite{H} have introduced Yamabe flow which describes a family of Riemannian metrics $g(t)$ subject to the evolution equation $\frac{\partial}{\partial t}g=-R(g)\,g$, where $R(g)$ denotes the scalar curvature corresponding to the metric $g$.
Hamilton proved local in time existence of Yamabe flows on compact manifolds without boundary.
Asymptotic behaviour of the Yamabe flow was subsequently analysed by B. Chow \cite{C}, R. Ye \cite{Y94}, Schwetlick and M. Struwe \cite{SS} and S. Brendle \cite{B}. The discrete Morse flow method for 2-d Yamabe flow was developed in \cite{MW}.
The theory of Yamabe flows on non-compact manifolds was addressed by
Ma and An \cite{AM}. Daskalopoulos and Sesum \cite{Daskalopoulos2013} analysed the profiles of self-similar solutions (Yamabe solitons).
More recently, Bahuaud and Vertman \cite{Bahuaud2014,Bahuaud2016} constructed Yamabe flows on spaces with incomplete edge singularities such that the singular structure is preserved along the Yamabe flow.
{Choi}, {Daskalopoulos}, and {King} \cite{Choi2018} were able to find solutions to the Yamabe flow on the Euclidean space $\mathbb{R}^n$ which develop a type II singularity in finite time. In the interesting work \cite{GT}, Gregor Giesen and Peter M. Topping obtained remarkable results about Yamabe flow incomplete surfaces. In \cite{S1, S2}, assuming that the initial metric is conformally hyperbolic with conformal factor and scalar curvature bounded from above, Schulz had obtained existence of instantaneously complete Yamabe flows on hyperbolic space of arbitrary dimension $n\geq3$.  The study of Yamabe flow on $R^n$ may be included in the class of porous-media equations \cite{Ar} \cite{DK} \cite{HP}.

Given an $n$-dimensional complete Riemannian manifold $(M^n,g_0)$, $n\geq 3$.
The Yamabe flow on  $(M^n,g_0)$ is a family of Riemannian metrics $\{g(\cdot, t)\}$ on $M$  defined by the
evolution equation
\begin{equation}\label{yamabe_flow_curvature}
\left\{
\begin{array}{ll}
         \frac{\partial g}{\partial t}=-Rg \quad &\text{in}\ M^n\times[0,T),\\
                  g(\cdot,0)=g_0 &\text{in}\ M^n,
\end{array}
\right.
\end{equation}
where  $R$ is the
scalar curvature of the metric
$ g:=g(\cdot,t)=u^{\frac{4}{n-2}}g_0, $
where
$u:M^n\to \mathbb{R}^+$ is a positive smooth function on $M^n$. Let $p=\frac{n+2}{n-2}$, $L_{g_0}u=\Delta_{g_0}u-aR_{g_0}u$ and
$a=\frac{n-2}{4(n-1)}$.
By changing time by a constant scale,
(\ref{yamabe_flow_curvature}) can be written in the equivalent form
\begin{equation}\label{yamabe_flow_u}
\left\{
\begin{array}{ll}
         \frac{\partial u^p}{\partial t}=L_{g_0}u, \quad &\text{in}\ M^n\times[0,T),\\
                  u(\cdot,0)=1, &\text{in}\ M^n.
\end{array}
\right.
\end{equation}

To understand the local existence result of the Yamabe flow on the Riemannian manifold $(M,g_0)$, we may choose a base metric $g_M$ on $M$ and write the Yamabe flow equation \cite{S1} as follows. Let $g(t)=w(x,t)g_M$ with $w=w(x,t)>0$ on $M$. Then the Yamabe flow equation is
$$
\frac{1}{m-1}w_t=-\frac{wR}{n-1}=-\frac{R_0}{n-1}+\frac{\Delta_{g_{M}} w}{w}+\frac{(n-6)}{4}\frac{|\nabla w|_{g_{M}}^2}{w^2},
$$
with $w(0)=w_0$.
We may denote the terms involving $w$ in the equation above by
$$
B[w]:=(n-1)\Bigl(-\frac{R_0}{n-1}+\frac{\Delta_{g_{M}} w}{w}+\frac{(n-6)}{4}\frac{|\nabla w|_{g_{M}}^2}{w^2}\Bigr).
$$
We shall apply inverse function theorem to this form of equation to get local existence of solutions.

Before presenting the main result of this paper, we need the following two definitions.
The first one is the definition of asymptotically flat (AF or ALE) manifold of order $\tau>0$ (\cite{S} \cite{LP}).
\begin{Def}\label{AE_def}
A Riemannian manifold $M^n$, $n\geq 3$, with $C^{\infty}$ metric $g$ is called asymptotically flat of order $\tau$
if there exists a decomposition $M^n= M_0\cup M_{\infty}$
(for simplicity we deal only with the case of one end and the case of multiple
ends can be dealt with similarly) with $M_0$ compact and a diffeomorphism $M_{\infty}\cong \mathbb{R}^n-B(o,R_0)$ for some constant $R_0 > 0$ such that
\begin{align}\label{AE}
g_{ij} -\delta_{ij}\in C^{2+\alpha}_{-\tau}(M)
\end{align}
(defined in Definition \ref{elliptic_wss} below) in the coordinates $\{x^i\}$ induced on $M_{\infty}$. And the coordinates $\{x^i\}$ are called asymptotic coordinates.
\end{Def}

The second one is about the fine solution to Yamabe flow (\cite{CZ} \cite{M3}).
\begin{Def}\label{fine}
We say that $u(x,t)\in C^1(M\times [0,t_{max})$ is a fine function
if $0<\delta\leq u(x,t)\leq C$ for $0\leq t\leq T$ with any $0<T<t_{max}$ and
$\sup\limits_{M^n\times [0,T]}|\nabla_{g_0} u(x,t)|\leq C$.
We say that $u(x,t)\in C^1(M\times [0,t_{max})$ is a fine solution of the Yamabe flow, $0\leq t<t_{max}$, on a complete manifold $(M^n,g_0)$ if it is a fine function solution to the Yamabe flow
and $\sup\limits_{M^n\times [0,T]}|Rm(g)|(x,t)\leq C$ for any $T<t_{max}$, such that either $\lim\limits_{t\to t_{max}}\sup\limits_{M}|Rm|(\cdot,t)=\infty$ for $t_{max}<\infty$ or $t_{max}=\infty$, where $Rm(g)$ is the Riemannian
curvature of the metric $g:=g(t)=u^{4/(n-2)}g_0$.
\end{Def}
We remark that in the language from (2.5.2) in \cite{SY} ( see also \cite{M4}), the fine solution is uniformly quasi-isometric to the initial metric $g_0$ on every interval $[0,T]$ for $0<T<t_{max}$.

Our main result is in below.
\begin{Thm}\label{global} Given an $n$-dimensional asymptotically flat manifold $(M^n,g_0)$ of any order $\tau>\frac{n-2}{2}$. There exists an unique global Yamabe flow $g(x,t)=u(x,t)^{4/(n-2)}g_0$ with the initial metric $g(0)=g_0$ and
 the solution $u(x,t)$, $0\leq t<t_0<\infty$, is a fine solution to the Yamabe flow (\ref{yamabe_flow_u}) and the flow preserves the AF property of the initial metric. In other word, for $v=1-u$, we have $v(x,t)\in C^{2+\alpha}_{-\tau}(M)$ and $g_{ij}(x,t)-\delta_{ij}\in C^{2+\alpha}_{-\tau}(M)$ for $t\in [0,t_{max})$.
\end{Thm}

The uniqueness part follows from the standard argument and we shall omit the detail. We shall present a general local existence result in Theorem \ref{short_existence} in section \ref{sect2}.

As a direct application of the computation as showed in \cite{CZ} and \cite{M}, we have
\begin{Thm}\label{main_5}
Let $u(x,t)$, $0\leq t< t_0<\infty$, be the fine solution to the Yamabe flow (\ref{yamabe_flow_u}) on an $n$-dimensional asymptotically flat manifold $(M^n,g_0)$ of order $\tau>\frac{n-2}{2}$ with $u(0)=1$. Assume that $R_{g_0}\geq 0$ and $R_{g_0}\in L^{1}(M)$, where $R_{g_0}$ is the scalar curvature of $g_0$.
Denoted by  $g(t)=u^\frac{4}{n-2}g_0$.
Then for $n=3,4,$ or $5$, ADM mass $m(g(t))$ (see \cite{SY} \cite{LP} or below for the definition) is well-defined under the Yamabe flow (\ref{yamabe_flow_u}) for $0\leq t<\infty$ (i.e. ADM mass is independent of the choices of the coordinates),i.e,  $m(g(t))\equiv m(g_0)$.
\end{Thm}

Recall here that the ADM mass of $n$-dimensional AF Riemannian manifolds \cite{LP} is defined as
\begin{align}\label{ADM_Mass}
m(g)=\lim\limits_{r\to\infty}\frac{1}{4\omega}\int_{S_r}(\partial_j g_{ij}-\partial_ig_{jj})dS^i,
\end{align}
where $\omega$ denotes the volume of unit sphere in $\mathbb{R}^n$, $S_r$ denotes the Euclidean sphere with radius $r$ and $dS^i$ is the normal surface volume element to $S_r$ with respect to Euclidean metric.
Similar results for Ricci flow were established in \cite{BW} and \cite{DM}.

With an application of Theorem 5 in \cite{M3}, we get the convergent result of the global Yamabe flow below.

\begin{Thm}\label{main_6}
Let $u(x,t)$, $0\leq t<\infty$, be the global solution to the Yamabe flow (\ref{yamabe_flow_u}) on an $n$-dimensional asymptotically flat manifold $(M^n,g_0)$ of order $\tau>\frac{n-2}{2}$ with $u(0)=1$. Assume that $R_{g_0}\geq 0$ and there exists a bounded sub-solution $w_0$ to the Poisson equation
$$
L_{g_0}w_0=\Delta_{g_0}w_0-aR_{g_0}w_0\geq 0, \ \ in \ M.
$$
Then the Yamabe flow $g(t)$ converges in $C^\infty_{loc}(M)$ to a Yamabe metric of  scalar curvature zero.
\end{Thm}

We now recall the definition of weighted spaces (see \cite{LP}) for elliptic operators on asymptotically flat manifolds.
\begin{Def}\label{elliptic_wss}
Suppose $(M^n,g)$ is an $n$-dimensional asymptotically flat manifold with asymptotic coordinates $\{x^i\}$. Denote $D^j_x v=\sup\limits_{|\alpha|=j}|\frac{\partial^{|\alpha|}}{\partial x_{i_1}\cdots\partial x_{i_j}}v|$.
Let $r(x)=|x|$ on $M_{\infty}$ (defined in Definition \ref{AE_def}) and extend $r$ to a smooth positive function on all of $M^n$. For $q\geq 1$ and $\beta\in \mathbb{R}$,
the weighted Lebesgue space $L^q_{\beta}(M)$ is defined as the set of locally integrable functions $v$ with the norm given by
$$ ||v||_{L^q_\beta(M)}=\left\{
                      \begin{array}{ll}
                        (\int_{M}|v|^q r^{-\beta q-n}dx)^{\frac{1}{q}}, & \hbox{$q<\infty$;} \\
                        ess \sup\limits_{M} (r^{-\beta}|v|), & \hbox{$q=\infty$.}
                      \end{array}
                    \right.
$$
 Then the weighted Sobolev space $W^{k,q}_\beta(M)$ is defined as the set of functions $v$ for which $|D^j_xv|\in L^q_{\beta-j}(M)$ with the norm
$$
||v||_{W^{k,q}_\beta(M)}=\sum\limits^k_{j=0}||D^j_x v||_{L^q_{\beta-j}(M)}.
$$
For a nonnegative integer $k$, the weighted $C^k$ space $C^k_{\beta}(M)$ is defined as the set of $C^k$ functions $v$ with the norm
$$
||v||_{C^k_\beta(M)}=\sum\limits_{j=0}^k\sup\limits_{M} r^{-\beta+j}|D^j_xv|.
$$
The weighted H\"{o}lder space $C^{k+\alpha}_{\beta}(M)$ is defined as the set of functions $v\in C^{k}_{\beta}(M)$ with the norm
$$
||v||_{C^{k+\alpha}_\beta(M)}=||v||_{C^k_\beta(M)}+\sup\limits_{x\neq y\in M}\min(r(x),r(y))^{-\beta+k+\alpha}\frac{|D^k_xv(x)-D^k_xv(y)|}{|x-y|^{\alpha}}.
$$
\end{Def}

We end the introduction with a brief outline of the paper. We discuss the local existence theory of Yamabe flow on a complete Riemannian manifold with bounded scalar curvature in section \ref{sect2}, this part may be well-known to experts. In section \ref{sect3}, we obtain the global Yamabe flows on AF manifolds. We show or give an outline proof of Theorems \ref{main_5} and \ref{main_6}.  In the appendix section \ref{sect4}, we discuss the general version of maximum principle, which may be used in the space decaying argument of the Yamabe flows on AF manifolds.

\section{Yamabe flow: local existence}\label{sect2}

Let $(M,g_M)$ be a complete Riemannian manifold of dimension $n=dim(M)$. Given an initial metric $g_0=w_0g_M$ where $w_0>0$ is a fine function on $M$. Let
$R_0=R(g_0)$ be the scalar curvature of the initial Riemannian metric $g_0$. The following local existence results of the solutions to the Yamabe flow  \eqref{eqn:Yamabe-flow} may be well-known for experts \cite{AM} (see also Theorem 2.4 in \cite{CZ}), but it is new.

\begin{Thm}\label{short_existence} Let $(M,g_M)$ be an $n$-dimensional complete manifold with bounded scalar curvature and let $g_0=w_0g_M$, where $w_0>0$ is a fine function on $M$.
Then Yamabe flow (\ref{eqn:Yamabe-flow}) below with initial metric $g_0$ has
a smooth solution on a maximal time interval $[0,T_{max})$ with $T_{max}>0$ such that either $T_{max}=+\infty$ or the evolving metric contracts to a point at finite time $T_{max}$.
\end{Thm}

Since the assumption above is weaker than previous existence result, one can not expects the uniqueness of the Yamabe flow.
We shall use the formulation from the interesting paper \cite{S1}.
The plan of the proof is to get the local existence result of the Yamabe flow on the Riemannian manifold $(M,g_M)$ by considering the evolution equation in the following form (\cite{AM})
\begin{align}\label{eqn:Yamabe-flow}
\frac{1}{n-1}w_t=-\frac{wR}{n-1}=-\frac{R_0}{n-1}+\frac{\Delta_{g_{M}} w}{w}+\frac{(n-6)}{4}\frac{|\nabla w|_{g_{M}}^2}{w^2},
\end{align}
with $w(0)=w_0$.
We denote the terms of right side of the equation \eqref{eqn:Yamabe-flow} by
\begin{align*}
B[w]:=(n-1)\Bigl(-\frac{R_0}{n-1}+\frac{\Delta_{g_{M}} w}{w}+\frac{(n-6)}{4}\frac{|\nabla w|_{g_{M}}^2}{w^2}\Bigr).
\end{align*}

We first set up the local existence result on any bounded domain with uniform time interval.
Given a smooth, bounded domain $\Omega\subset M$ and $T>0$, we may assume that $-R_0\geq n c$ for some constant $c$ and we consider the problem
\begin{align}\label{eqn:pde}
\left\{
\begin{aligned}
\frac{\partial w}{\partial t}&=B[w]
&&\text{ in $\Omega\times[0,T]$, } \\
w&=\phi
&&\text{ on $\partial\Omega\times[0,T]$, }  \\[.5ex]
w&=w_0
&&\text{ on $\Omega\times\{0\}$}
\end{aligned}\right.
\end{align}
for given $0<w_0\in C^{2,\alpha}(\overline{\Omega})$ and $\phi\in C^{2,\alpha;1,\frac{\alpha}{2}}(\partial\Omega\times[0,T])$ satisfying $\phi(\cdot,t)=w_0$ on $\partial\Omega\times[0,T]$
Since $w_0$ and $R_{g_0}$ are bounded on the compact set $\partial\Omega$ and $w_0>0$, the nonlinear term $B[w]$ is well-defined at the initial time.

By the standard parabolic theory \cite{Ladyzenskaja1967} we may solve
 the linear parabolic problem
\begin{align}\label{eqn:pde-linear-u}
\left\{\begin{aligned}
\frac{1}{n-1}\frac{\partial\tilde{u}}{\partial t}-\frac{\Delta_{g_{M}}\tilde{u}}{w_0}
-\frac{(n-6)}{4}\frac{<\nabla\tilde{u},\nabla w_0>_{g_{M}}}{w_0^2}
&=-\frac{R_0}{n-1}
&&\text{ in $\Omega\times[0,T]$, } \\
\tilde{u}&=\phi
&&\text{ on $\partial\Omega\times[0,T]$, }  \\[.5ex]
\tilde{u}&=w_0
&&\text{ on $\Omega\times\{0\}$. }
\end{aligned}\right.
\end{align}
to get the solution $\tilde{u}$. Since $\Omega$ is bounded and since $w_0>0$ in $M$,
there exists some $\delta>0$ depending on $\Omega$ and $w_0$ such that $w_0\geq \delta$ in $\Omega\times [0,T]$.
Therefore, equation \eqref{eqn:pde-linear-u} is uniformly parabolic with regular coefficients and the initial-boundary conditions are satisfied.
According to linear parabolic theory \cite{Ladyzenskaja1967}[IV.5, Theorem 5.2], problem \eqref{eqn:pde-linear-u} has a unique solution $\tilde{u}\in C^{2,\alpha;1,\frac{\alpha}{2}}(\overline{\Omega}\times[0,T])$.
Since $w_0>0$ and $\phi(\cdot,t)=w_0$ for all $t\in[0,T]$, the parabolic maximum principle  applied to $\tilde{u}(\cdot,t)$ implies $\tilde{u}\geq\varepsilon$ on $\Omega\times[0,T]$ for some $\varepsilon>0$ depending on $\Omega$ and $\tilde{u}$.
For the short time $t>0$, we want to get the solution $w$ to \eqref{eqn:pde} which will be close to the function $\tilde{u}$.

We shall use the inverse function theorem to construct the short time solution to  \eqref{eqn:pde}  on any bounded domain.
\begin{Lem}[Short-time existence on bounded domains]
\label{lem:shorttimeexistence}
Let $\Omega\subset M$ be a smooth bounded domain in $(M, g_M)$.
Then there exists $T>0$ such that problem \eqref{eqn:pde} has a unique solution.
\end{Lem}

\begin{proof}
We shall construct a solution $w$ to \eqref{eqn:pde} is of the form $w=\tilde{u}+v$, where $\tilde{u}$ solves \eqref{eqn:pde-linear-u} and
\begin{align}\label{eqn:pde-v}
\left\{\begin{aligned}
\frac{\partial v}{\partial t}
&=B[\tilde{u}+v]-\frac{\partial \tilde{u}}{\partial t}
&&\text{ in $\Omega\times[0,T]$, } \\
v&=0
&&\text{ on $\partial\Omega\times[0,T]$, }  \\[.5ex]
v&=0
&&\text{ on $\Omega\times\{0\}$. }
\end{aligned}\right.
\end{align}
For the H\"older exponent $0<\alpha<1$, we define the working space
\begin{align*}
X&:=\{v\in C^{2,\alpha;1,\frac{\alpha}{2}}(\overline{\Omega}\times[0,T])
\mid{}v=0\text{ on }(\Omega\times\{0\})\cup(\partial\Omega\times[0,T])\},
\\[1ex]
Y&:=\{f\in
C^{0,\alpha;0,\frac{\alpha}{2}}(\overline{\Omega}\times[0,T])
\mid{}f=0\text{ on }\partial\Omega\times\{0\}\}.
\end{align*}
Notice that the map
$F: X\to Y$,
\begin{align*}
F: v\mapsto \frac{\partial}{\partial t}(\tilde{u}+v)-B[\tilde{u}+v].
\end{align*}
 is well-defined because the initial-boundary conditions imply that at every $p\in\partial\Omega$ for every $v\in X$, we have
\begin{align*}
(F v)(p,0)&=\Bigl(\frac{\partial\tilde{u}}{\partial t}-B[\tilde{u}]\Bigr)(p,0)
=\Bigl(\frac{\partial\phi}{\partial t}(\cdot,0)-B[u_0]\Bigr)(p)=0.
\end{align*}
The linearization of $B[\tilde{u}]$ around $\tilde{u}\in C^{2,\alpha;1,\frac{\alpha}{2}}(\overline{\Omega}\times[0,T])$ gives the linear operator
\begin{align*}
\breve{L}(\tilde{u})&=(n-1)\Bigl(-\frac{\Delta_{g_{M}}\tilde{u}}{\tilde{u}^2}
-\frac{(n-6)}{2}\frac{|\nabla\tilde{u}|_{g_{M}}^2}{\tilde{u}^3}+
\frac{(n-6)}{2\tilde{u}^2}<\nabla\tilde{u},\nabla\,\cdot\,>_{g_{M}}
+\frac{\Delta_{g_{M}}}{\tilde{u}}
\Bigr).
\end{align*}

We claim that the map $F$ is Fr\'echet differentiable at $0\in X$. The reason is below.
First, the map $F$ is G\^ateaux differentiable at $0\in X$ with derivative
\begin{align*}
D F(0)\colon X&\to Y
\\
w&\mapsto \frac{\partial}{\partial t}w-L(\tilde{u})w.
\end{align*}
Second, the mapping $u\mapsto \breve{L}(u)$ is continuous near $\tilde{u}$ because $\tilde{u}$ is bounded away from zero.
Hence, $DF(0)$ is the Fr\'{e}chet-derivative of $S$ at $0\in X$.
Note that the linear operator $\frac{\partial}{\partial t}-\breve{L}(\tilde{u})$ is uniformly parabolic.

Let $f\in Y$ be an arbitrary element.
By definition, $0=f(\cdot,0)$ on $\partial\Omega$. We consider the linear parabolic problem
\begin{align}\label{eqn:pde-linear}
\left\{\begin{aligned}
\frac{\partial w}{\partial t}-\breve{L}(\tilde{u})w&=f
&&\text{ in $\Omega\times[0,T]$, } \\
w&=0
&&\text{ on $\partial\Omega\times[0,T]$, } \\[.5ex]
w&=0
&&\text{ on $\Omega\times\{0\}$. }
\end{aligned}\right.
\end{align}
As before, linear parabolic theory guarantees that \eqref{eqn:pde-linear} has a unique solution $w\in X$.
Hence, the continuous linear map $DF(0)\colon X\to Y$ is invertible.

By the Inverse Function Theorem, $F$ is invertible in some neighborhood $V_0\subset Y$ of $F(0)$.
Claim that $V_0$ contains an element $h$ such that $h(\cdot,t)=0$ for $0\leq t\leq\varepsilon$ and sufficiently small $\varepsilon>0$.
Fix $f:=F(0)=\frac{\partial}{\partial t}\tilde{u}-B[\tilde{u}]$.
Choose $\eta: [0,T]\to[0,1]$, a smooth cutoff function such that
\begin{align*}
\eta(t)&=\begin{cases}
0, &\text{ for } t\leq\varepsilon, \\
1, &\text{ for } t>2\varepsilon,
\end{cases}&
0&\leq\frac{d\eta}{d t}\leq\frac{3}{\varepsilon}.
\end{align*}
Note that $\eta f\in V_0$ for sufficiently small $\varepsilon>0$. In fact,
since $\tilde{u}$ is smooth in $\overline{\Omega}\times[0,T]$, we have $f\in C^{1}(\overline{\Omega}\times[0,T])$.
Noting at $t=0$, we have
\begin{align}\label{est:f-hoelder1}
f(\cdot,0)&=\frac{\partial\tilde{u}}{\partial t}(\cdot,0)-B[w_0]=0 \quad\text{ on $\overline{\Omega}$, }
\end{align}
we may estimate
\begin{align}\label{eqn:sC1}
|{f(\cdot,s)}|=|f(\cdot,s)-f(\cdot,0)|\leq s|{f}|_{C^1(\overline{\Omega}\times[0,T])}.
\end{align}
Take $t,s\in[0,T]$ and $t>s$.
For $s>2\varepsilon$, we have
$$(f-\eta f)(\cdot,s)=(f-\eta f)(\cdot,t)=0.$$
Then we may assume $s\leq2\varepsilon$.
In this case we may estimate the time difference of the function $(f-\eta f)$ in the following way.
\begin{align}\nonumber
&|(f-\eta f)(\cdot,t)-(f-\eta f)(\cdot,s)|
\\ \nonumber
&\leq |f(\cdot,t)-f(\cdot,s)|
+|\eta f(\cdot,t)-\eta f(\cdot,s)|
\\ \nonumber
&\leq
\bigl(1+|\eta(t)|\bigr)|f(\cdot,t)-f(\cdot,s)|
+|f(\cdot,s)||\eta(t)-\eta(s)|
\\ \nonumber
&\leq2 |f|_{C^1}|t-s|+s|{f}|_{C^1}|\eta'|_{C^0}|{t-s}|
\\
&\leq\bigl(2+s\tfrac{3}{\varepsilon}\bigr)|f|_{C^1}|{t-s}|
\\ \label{est:f-hoelder3}
&\leq 8 |{f}|_{C^1}|{t-s}|.
\end{align}
By \eqref{est:f-hoelder1}, we may reduce the special case $s=0$ to the bound
\begin{align}\label{eqn:20190125-3}
|(f-\eta f)(\cdot,t)|
&\leq8t |{f}|_{C^1}.
\end{align}
Since the left-hand side of \eqref{eqn:20190125-3} vanishes for $t>2\varepsilon$, we may have
\begin{align*}
|f-\eta f|_{C^0}
&\leq 16\varepsilon |f|_{C^1}.
\end{align*}
If $|t-s|<\varepsilon$, the estimate \eqref{est:f-hoelder3} implies that
\begin{align*}
|(f-\eta f)(\cdot,t)-(f-\eta f)(\cdot,s)|
&\leq 8\varepsilon^{1-\frac{\alpha}{2}}|f|_{C^1}|{t-s}|^{\frac{\alpha}{2}}.
\end{align*}
If $|{t-s}|\geq\varepsilon$, we may replace the estimate by the fact that
\begin{align*}
|(f-\eta f)(\cdot,t)-(f-\eta f)(\cdot,s)|
&\leq 2|f-\eta f|_{C^0}
\\ \nonumber
&\leq 32\varepsilon |{f}|_{C^1}
\\ \nonumber
&\leq 32\varepsilon^{1-\frac{\alpha}{2}}|f|_{C^1}|t-s|^{\frac{\alpha}{2}}.
\end{align*}
Then, $$[f-\eta f]_{\frac{\alpha}{2},t}\leq32\varepsilon^{1-\frac{\alpha}{2}}|f|_{C^1}.$$

For the estimation of the spatial H\"older seminorm, we may obtain a similar estimate from \eqref{eqn:sC1} and  estimate the space difference:
\begin{align*}
&|{(f-\eta f)(x,t)-(f-\eta f)(y,t)}|
\\ \nonumber
&\leq |{1-\eta(t)}||{f(x,t)-f(y,t)}|^{\alpha}|{f(x,t)-f(y,t)}^{1-\alpha}|
\\ \nonumber
&\leq |{f}|_{C^1}^{\alpha}\,d(x,y)^{\alpha}\bigl(4\varepsilon |{f}|_{C^1}\bigr)^{1-\alpha}
=(4\varepsilon)^{1-\alpha}|{f}|_{C^1}\,d(x,y)^{\alpha},
\end{align*}
where $d(x,y)$ is the Riemannian distance between $x$ and $y$ in $(M,g_{M})$.
Then, $$|{f-\eta f}|_{Y}\leq C\varepsilon^{\beta-\alpha}|{f}|_{C^1}.$$
This implies that $\eta f$ belongs to the neighborhood $V_0$ of $f$ if  $\varepsilon>0$ is sufficiently small.
By the construction above, $F^{-1}(\eta f)$ is a solution to \eqref{eqn:pde-v} in $\Omega\times[0,\varepsilon]$.
Setting $T=\varepsilon>0$, we then obtain the desired result.
\end{proof}

We are now going to prove Theorem \ref{short_existence}.
\begin{proof}
We now may obtain the local in time solution to \eqref{eqn:Yamabe-flow} on the whole Riemannian manifold $(M,g_M)$.
Recall that we have assumed the scalar curvature of $g_0=w_0g_M$ is bounded.  Let $\Omega_1\subset\Omega_2\subset\cdots$ be the smooth compact domain exhaustion of $M$ (the existence of such domain exhaustion was used in \cite{D}). Recall that $p=\frac{n+2}{n-2}$, $L_{g_0}v=\Delta_{g_0}v-aR_{g_0}v$ and
$a=\frac{n-2}{4(n-1)}$. We may write $g_0=v_0^{4/(n-2)}g_M$ for $w_0=v_0^{4/(n-2)}$ and look for for solution of the form
$$g(x,t)=\check{u}^{4/(n-2)}g_0=(\check{u}{v_0})^{4/(n-2)}g_M=v^{4/(n-2)}g_M=ug_M.
$$
Then the Yamabe flow equation may be written as
$$
  \frac{\partial v^p}{\partial t}=L_{g_M}v, \quad \ x\in M, \ t>0,
$$
with the initial data
$v(0)=v_0$. Recall that $L_{g_M}v =\Delta_{g_M}v-aR_{g_M}v$ in $M$.
 For shortening the notation, we may assume $v_0=1$ and then $g_0=g_M$.
Then, the solution $\check{u}(x,t)$ to Yamabe flow \eqref{eqn:Yamabe-flow} may be obtained by a sequence of approximation solutions
$u_m(x,t)=\check{u}_m(x,t)^{4/(n-2}$ obtained above. Note that $\check{u}_m(x,t)$ satisfies
\begin{equation}\label{yamabe_flow_u_local1}
\left\{
\begin{array}{ll}
         \frac{\partial \check{u}^p_m}{\partial t}=L_{g_0}\check{u}_m, \quad &\ x\in\Omega_m, \ t>0,\\
                   \check{u}_m(x,t)>0, \quad &\ x\in\Omega_m, \ t>0,\\
                   \check{u}_m(x,t)=1, \quad &\ x\in \partial \Omega_m, \ t>0,\\
                  \check{u}_m(\cdot,0)=1, \quad &\ x\in\Omega_m.\\
\end{array}
\right.
\end{equation}
 Since  $\check{u}_m(x,t)=1$ is bounded on $\partial \Omega_m$, by the maximum principle, we may conclude that
\begin{align*}
\max\limits_{\Omega_m} \check{u}_m(t)\leq (1+\frac{n-2}{(n-1)(n+2)}\sup\limits_{M^n}|R_{g_0}|t)^{\frac{n-2}{4}}.
\end{align*}
and
\begin{align*}
\min\limits_{\Omega_m} \check{u}_m(t)\geq (1-\frac{n-2}{(n-1)(n+2)}\sup\limits_{M^n}|R_{g_0}|t)^{\frac{n-2}{4}}.
\end{align*}
We see that $\check{u}_m(t)$ has an uniformly upper bound on $[0,t_0)$ for any $t_0>0$ and uniformly positive lower bound on $[0,\frac{(n-1)(n+2)}{2(n-2)\sup\limits_{M^n}|R_{g_0}|}]$. Let $T=\frac{(n-1)(n+2)}{2(n-2)\sup\limits_{M^n}|R_{g_0}|}$. Then every local solution $\{u_m\}$ is well-defined on the time interval $[0,T]$.

Applying  Trudinger's estimate \cite{Trudinger1968} (or the Krylov-Safonov estimate) and Schauder estimate of parabolic equations to \eqref{eqn:Yamabe-flow} on any ball $B_{g_0}(p,r_0)\subset (M,g_0)$, we have
$||u_m||_{C^{2+\alpha,1+\frac{\alpha}{2}}(B_{g_0}(p,r_0)\times [0,T])}\leq C$, where $C$ is independent of the point $p$.
Using the diagonal subsequence of $\{u_m\}$, we may extract a $C^{2+\alpha,1+\frac{\alpha}{2}}_{loc}$ convergent sequence with its positive limit $u(x,t)=\check{u}^{\frac{4}{n-2}}$ on whole $M$, which is the desired local in time solution to \eqref{eqn:Yamabe-flow}.
Since $g(\cdot,t)=\check{u}^{\frac{4}{n-2}}g_0$, we also have $\sup\limits_{B_{g_0}(p,r_0)\times [0,T]}|Rm(x,t)|\leq C$.
With this understanding, we may extend the solution to the maximal time solution as we wanted.

This completes the proof of the result.
\end{proof}

\section{global Yamabe flows on ALE manifolds}\label{sect3}

  Assume that $(M,g_0)$ is an ALE manifold. We shall show that the Yamabe flow exists globally.  Assume by contrary that the maximal time of Yamabe flow is finite, i.e., $T_{max}<\infty$. Then according to AF property of the solution in Theorem 5.1 \cite{CZ} (and its argument is based on the generalized maximum principle Theorem \ref{Ecker_Huisken} as showed in appendix), we know that there exists a compact set $S\subset M$ such that
  $$
  \frac12\leq u(x,t)\leq 3/4, \ \ \forall (x,t)\in (M\setminus S)\times [0,T_{max})
  $$
  and there is a point $x\in S$ such that $u$ is non-trivial in a neighborhood of $x$.
  As in  \cite{Y94}, using DiBennedetto's estimate we may extend the solution $u$ continuously to $T_{max}$. Using Ye's argument, we know that $u$ can not have any zero point in $S$ at $T_{max}$. Therefore, there exist two positive constants $c_1$ and $c_2$ such that
  $$
  c_1\leq u(x,t)\leq c_2, \ \ \forall (x,t)\in M\times [0,T_{max}] \ \text{uniformly} .
  $$
  Then we may use the standard parabolic theory \cite{Trudinger1968} to extend the solution beyond $T_{max}$, which is a contradiction with $T_{max}<\infty$. The decaying property of the Yamabe flow follows from Theorem 5.1 in \cite{CZ}.
  This then completes the proof of Theorem \ref{global}.

The proof of  Theorem \ref{main_5} follows from the application of Theorem \ref{short_existence}, Theorem 5.1 in \cite{CZ}, and Theorem 6 in \cite{M}.

Since the proof of Theorem \ref{main_6} is by now easy to give and the proof is below.
\begin{proof}
 We choose $\delta>0$ small such that $\tilde{u}_0=\delta w_0<1$. Note that $\tilde{u}_0$ is the lower solution of the Yamabe flow.
  Let $\tilde{g}_0= \tilde{u}_0^{4/(n-2)}g_0$. Then  the scalar curvature of the metric $\tilde{g}_0$ is non-negative and we also have
$$
g_0\geq \tilde{g}_0.
$$
Note that along the Yamabe flow $g(t)\geq \tilde{g}_0$ and by the maximum principle we know that the scalar curvature $R=R(g(t))\geq 0$ on $M$. Since
$$
\frac{\partial g}{\partial t}=-Rg \leq 0,
$$
We then know that
the Yamabe flow $g(t)$ converges in $C^\infty_{loc}(M)$ to a Yamabe metric of  scalar curvature zero.
\end{proof}

\section{Appendix: the maximum principle}\label{sect4}
In this section, we present a generalized version of the maximum principle
(see Theorem 4.3 in \cite{EH} and Theorem 2.6 in \cite{CZ}), where they consider
the maximum principle for the parabolic equation $\frac{\partial }{\partial t}v - \Delta v\leq b\cdot \nabla v+cv$ or
 $\frac{\partial }{\partial t}v -  div(a \nabla v)\leq b\cdot \nabla v+cv$
on noncompact manifolds, where $\Delta$ and $\nabla$ depend on $g(t)$. Our maximum principle is about the more general equation
 $$m(x) \frac{\partial }{\partial t}v -  div(a \nabla v)\leq b\cdot \nabla v+cv, \ \ M\times [0,T)$$
where $m(x)$ is a positive regular function on $M$.

\begin{Thm}\label{Ecker_Huisken}
Suppose that the complete noncompact manifold $M^n$ with Riemannian metric $g(t)$ satisfies the uniformly volume
growth condition
\begin{align*}
    vol_{g(t)}(B_{g(t)}(p,r))\leq exp(k(1+r^2))
\end{align*}
for some point $p\in M$ and a uniform constant $k>0$ for all $t\in [0,T]$. Let $v$ be a differentiable function on $M\times (0,T]$ and
continuous on $M\times [0,T]$. Assume that $v$ and $g(t)$ satisfy

(i) The differential inequality
\begin{align*}
    m(x)\frac{\partial }{\partial t}v - div(a \nabla v)\leq b\cdot \nabla v+cv,
\end{align*}
where $m(x)$ is a positive continuous function on $M$ such that $0<m_0\leq m(x)\leq m_1$ for some constant $m_0>0$ and $m_1>0$, the vector field $b$ and the function $a$ and $c$ are uniformly bounded
\begin{align*}
    0<\alpha_1' \leq a\leq \alpha_1, \sup\limits_{M\times [0,T]} |b|\leq \alpha_2, \sup\limits_{M\times [0,T]} |c|\leq \alpha_3,
\end{align*}
for some constants $\alpha_1',\alpha_1,\alpha_2<\infty$. Here $\Delta$ and $\nabla$ depend on $g(t)$.

(ii) The initial data
\begin{align*}
    v(p,0)\leq 0,
\end{align*}
for all $p\in M$.

(iii) The growth condition
\begin{align*}
    \int^T_0(\int_{M}exp[-\alpha_4 d_{g(t)}(p,y)^2]|\nabla v|^2(y)d \mu_t)dt<\infty.
\end{align*}
for some constant $\alpha_4>0$.

(iv) Bounded variation condition in metrics in the sense that
\begin{align*}
    \sup\limits_{M\times[0,T]}|\frac{\partial}{\partial t}g(t)|\leq \alpha_5
\end{align*}
for some constant $\alpha_5<\infty$.

Then we have
$ v\leq 0 $
on $M\times [0,T]$.
\end{Thm}

\begin{Rk}\label{remark_max}
Note that the conditions (iii) and (iv) are satisfied if the sectional curvature of $g(t)$ and $\nabla v$ are uniformly bounded on $[0,T]$. There  are many versions of maximum principles, one may prefer to \cite{A} and \cite{BK}.
\end{Rk}

 \textbf{Proof of Theorem \ref{Ecker_Huisken}:}
Fix $K_0>0$ large. We choose $\theta>0$ and let
$$
h(y,t)=-\frac{\theta d^2_{g(t)}(p,y)}{4(2\eta-t)},0<t<\eta,
$$
where $d_{g(t)}(p,y)$ is the distance between $p$ and $y$ at time $t$ and
$0<\eta<\min(T,\frac{1}{64K_0},\frac{1}{32\alpha_4},\frac{1}{4\alpha_5})$. Then
\begin{align*}
\frac{d}{dt}h=-\frac{\theta d^2_{g(t)}(p,y)}{4(2\eta-t)^2}-\frac{\theta d_{g(t)}(p,y)}{2(2\eta-t)}\frac{d}{dt}d_{g(t)}(p,y).
\end{align*}
By (iv), we have
\begin{align*}
|\frac{d}{dt}d_{g(t)}(p,y)|\leq \frac{1}{2}\alpha_5 d_{g(t)}(p,y).
\end{align*}
Then we have that
\begin{align*}
    \frac{d}{dt}h\leq -\theta^{-1}|\nabla h|^2+\theta^{-1}\alpha_5|\nabla h|^2 (2\eta-t),
\end{align*}
We choose $\theta=\frac{1}{4\alpha_1}$. Using $\eta\leq \frac{1}{4\alpha_5}$ we have
\begin{align}\label{h_estiamte}
m_0\frac{d}{dt}h+2a|\nabla h|^2\leq 0.
\end{align}
Let $K>0$, which will be a very large constant.
Taking $f_K=\max\{\min(f,K),0\}$ and $0<\epsilon<\eta$, we have
\begin{align*}
&\int^{\eta}_{\epsilon}e^{-\beta t}(\int_{M}\phi^2 e^h f_K(div(a\nabla f)-\frac{\partial f}{\partial t})d\mu_t)dt\\
\geq & -\alpha_2 \int^{\eta}_{\epsilon}e^{-\beta t}(\int_{M}\phi^2 e^h f_K|\nabla f|d\mu_t)dt\\
&-\alpha_3 \int^{\eta}_{\epsilon}e^{-\beta t}(\int_{M}\phi^2 e^h f_K fd\mu_t)dt
\end{align*}
for some smooth time independent compactly supported function $\phi$ on $M^n$, where $\beta>0$ will be chosen later. Then we have
\begin{align*}
0\leq &-\int^{\eta}_{\epsilon}e^{-\beta t}(\int_{M}\phi^2 e^h a <\nabla f_K,\nabla f_K>d\mu_t)dt\\
&-\int^{\eta}_{\epsilon}e^{-\beta t}(\int_{M}\phi^2 e^h f_K a<\nabla h,\nabla f>d\mu_t)dt\\
&-2\int^{\eta}_{\epsilon}e^{-\beta t}(\int_{M}\phi e^h f_K a<\nabla \phi,\nabla f>d\mu_t)dt\\
&-\int^{\eta}_{\epsilon}e^{-\beta t}(\int_{M}m(x)\phi^2 e^h f_K\frac{\partial f}{\partial t}d\mu_t)dt
+\alpha_3 \int^{\eta}_{\epsilon}e^{-\beta t}(\int_{M}\phi^2 e^h f_K fd\mu_t)dt \\
&+\alpha_2 \int^{\eta}_{\epsilon}e^{-\beta t}(\int_{M}\phi^2 e^h f_K|\nabla f|d\mu_t)dt\\
=&\textrm{I}+\textrm{II}+\textrm{III}+\textrm{IV}+\textrm{V}+\textrm{VI}.
\end{align*}
By Schwartz' inequality, we derive
\begin{align*}
\textrm{II}\leq \frac{1}{4}\int^{\eta}_{\epsilon}e^{-\beta t}(\int_{M}\phi^2 e^h a|\nabla f|^2d\mu_t)dt+\int^{\eta}_{\epsilon}e^{-\beta t}(\int_{M}\phi^2 e^h f_K^2 a|\nabla h|^2d\mu_t)dt,
\end{align*}
\begin{align*}
\textrm{III}\leq \frac{1}{2}\int^{\eta}_{\epsilon}e^{-\beta t}(\int_{M}\phi^2 e^h a|\nabla f|^2d\mu_t)dt+2\int^{\eta}_{\epsilon}e^{-\beta t}(\int_{M} e^h f_K^2 a|\nabla \phi|^2d\mu_t)dt,
\end{align*}
and
\begin{align*}
\textrm{VI}&\leq \frac{1}{4}\int^{\eta}_{\epsilon}e^{-\beta t}(\int_{M}\phi^2 e^h a|\nabla f|^2d\mu_t)dt+\alpha_2^2\int^{\eta}_{\epsilon}e^{-\beta t}(\int_{M} e^h f_K^2 \frac{1}{a}|\nabla \phi|^2d\mu_t)dt\\
&\leq \frac{1}{4}\int^{\eta}_{\epsilon}e^{-\beta t}(\int_{M}\phi^2 e^h a|\nabla f|^2d\mu_t)dt+\frac{\alpha_2^2}{\alpha_1'}\int^{\eta}_{\epsilon}e^{-\beta t}(\int_{M} e^h f_K^2 |\nabla \phi|^2d\mu_t)dt.
\end{align*}
Since
\begin{align*}
-e^hf_K\frac{\partial f}{\partial t}\leq -e^h f_K \frac{\partial f_K}{\partial t}+\frac{\partial }{\partial t}(e^hf_K(f_K-f)),
\end{align*}
and
$$f_K(f_K-f)\leq 0,$$
we obtain
\begin{align*}
&\qquad\textrm{IV}+\textrm{V}\\
&\leq -\frac{1}{2}\int^{\eta}_{\epsilon}e^{-\beta t}(\int_{M}m(x)\phi^2 e^h \frac{\partial f_K^2}{\partial t}d\mu_t)dt
+\int^{\eta}_{\epsilon}e^{-\beta t}(\int_{M}m(x)\phi^2 \frac{\partial }{\partial t}(e^h f_K(f_K-f))d\mu_t)dt\\
&
-\alpha_3 \int^{\eta}_{\epsilon}e^{-\beta t}(\int_{M}\phi^2 e^h f_K(f_K-f)d\mu_t)dt+\alpha_3 \int^{\eta}_{\epsilon}e^{-\beta t}(\int_{M}\phi^2 e^h f_K^2d\mu_t)dt.
\end{align*}
Moreover, we have
\begin{align*}
|\frac{d}{dt}(d\mu_t)|\leq n \alpha_5 d\mu_t
\end{align*}
by (iv). Now we choose $\beta>0$ such that $m_0\beta\geq 2n\alpha_5+4\alpha_3+4\frac{\alpha_2^2}{\alpha_1'}$. Then
\begin{align*}
&\qquad\textrm{IV}+\textrm{V}\\
&\leq -\frac{1}{2}e^{-\beta t}\int_{M}m(x)\phi^2 e^h f_K^2d\mu_t|_{t=\eta}
+\frac{1}{2}e^{-\beta t}\int_{M}m(x)\phi^2 e^h f_K^2d\mu_t|_{t=\epsilon}\\
&+\frac{1}{2}\int^{\eta}_{\epsilon}e^{-\beta t}(\int_{M}m(x)\phi^2 e^h f_K^2 \frac{\partial h}{\partial t}d\mu_t)dt-\frac{1}{4}m_0\beta\int^{\eta}_{\epsilon}e^{-\beta t}(\int_{M}\phi^2 e^h f_K^2 d\mu_t)dt\\
&
+e^{-\beta t}\int_{M}\phi^2 e^h f_K(f_K-f)d\mu_t|_{t=\eta}-e^{-\beta t}\int_{M}\phi^2 e^h f_K^2d\mu_t|_{t=\epsilon}.
\end{align*}
Combining the estimates of $\textrm{I}-\textrm{VI}$ and letting $\epsilon\to 0$, we obatin
\begin{align*}
0\leq &-\int^{\eta}_{0}e^{-\beta t}(\int_{M}\phi^2 e^h a |\nabla f_K|^2d\mu_t)dt
+\int^{\eta}_{0}e^{-\beta t}(\int_{M}\phi^2 e^h a |\nabla f|^2d\mu_t)dt\\
&+2\int^{\eta}_{0}e^{-\beta t}(\int_{M} e^h f_K^2 a|\nabla \phi|^2d\mu_t)dt-\frac{1}{2}e^{-\beta t}\int_{M}m(x)\phi^2 e^h f_K^2d\mu_t|_{t=\eta}.
\end{align*}
by $f_K\equiv 0$ at $t=0$ and (\ref{h_estiamte}). Now we choose $0\leq \phi\leq 1$ satisfying $\phi \equiv 1$ on $B_{g_0}(p,R)$, $\phi \equiv 0$ outside $B_{g_0}(p,R+1)$ and $|\nabla_{g_0} \phi|_{g_0}\leq 2$. Then we have
\begin{align*}
&\frac{1}{2}e^{-\beta \eta}\int_{B_{g_0}(p,R)}m(x)\phi^2 e^h f_K^2d\mu_t|_{t=\eta}\leq \int^{\eta}_{0}e^{-\beta t}(\int_{B_{g_0}(p,R+1)}\phi^2 e^h a (|\nabla f|^2-|\nabla f_K|^2)d\mu_t)dt\\
&+C(\alpha_5)\int^{\eta}_{0}e^{-\beta t}(\int_{B_{g_0}(p,R+1)\backslash B_{g_0}(p,R)} e^h f_K^2 a d\mu_t)dt,
\end{align*}
where $C(\alpha_5)$ is a constant only depending on $\alpha_5$. By $0<\eta<\min(\frac{1}{K_0},\frac{1}{32\alpha_4})$ and volume growth assumptions
on $M^n$, we have
\begin{align*}
\int^{\eta}_{0}e^{-\beta t}(\int_{B_{g_0}(p,R+1)\backslash B_{g_0}(p,R)} e^h f_K^2 a d\mu_t)dt\to 0,
\end{align*}
as $R\to \infty$. Then we derive
\begin{align*}
&\frac{1}{2}e^{-\beta \eta}\int_{M}\phi^2 e^h f_K^2d\mu_t|_{t=\eta}\leq \int^{\eta}_{0}e^{-\beta t}(\int_{M}\phi^2 e^h a (|\nabla f|^2-|\nabla f_K|^2)d\mu_t)dt.
\end{align*}
Letting $K\to\infty$, we conclude that
\begin{align*}
\frac{1}{2}e^{-\beta \eta}\int_{M}m(x)\phi^2 e^h (\max(f,0))^2d\mu_t|_{t=\eta}\leq 0,
\end{align*}
where $0<\eta<\min(T,\frac{1}{64K_0},\frac{1}{32\alpha_4},\frac{1}{4\alpha_5})$. That implies that $f\leq 0$ in $M^n\times[0,\eta]$. By the induction argument, we then have that
$f\leq 0$ in $M^n\times[0,T]$.
$\Box$

\end{document}